\documentclass[12pt,reqno]{amsart}
\usepackage{jrmacros}

\usepackage{mathtools}
\mathtoolsset{showonlyrefs}

\usepackage{appendix}

\numberwithin{equation}{section}

\theoremstyle{plain}
\newtheorem{theorem}{Theorem}[section]
\newtheorem{proposition}[theorem]{Proposition}

\newtheorem{conjecture}[theorem]{Conjecture}

\newtheorem{theor}{Theorem}
\newtheorem{conj}[theor]{Conjecture}

\newtheorem*{theorem*}{Theorem}
\newtheorem*{proposition*}{Proposition}
\newtheorem*{corollary*}{Corollary}
\newtheorem*{lemma*}{Lemma}
\newtheorem*{conjecture*}{Conjecture}

\theoremstyle{definition}
\newtheorem{definition}[theorem]{Definition}

\newtheorem{defn}[theor]{Definition}

\newtheorem*{definition*}{Definition}
\newtheorem*{example*}{Example}
\newtheorem*{question*}{Question}
\newtheorem*{philosophy*}{Philosophy}

\theoremstyle{remark}
\newtheorem{remark}[theorem]{Remark}

\newtheorem*{remark*}{Remark}

\def\ape{Ap\'{e}ry}

\def\zm{\zeta^{\cM}}
\def\G{\mathbb{G}}

\def\MM{\cZ}
\def\zm{\zeta^{\cM}}

\def\p{\mathbf{p}}

\def\Ai{\Q_{p\to\infty}}

\def\M{\cH}

\def\MM{\cA}

\def\MMh{\hat{\MM}}

\def\per{\widehat{per}}

\def\zm{\zeta^{\fm}}

\def\zmm{\zeta^{\fa}}

\def\PP{\cP}

\def\PPm{\PP^{\fm}}

\def\mhs{\cM}

\def\fil{\mathrm{Fil}}
\def\p{\mathbf{p}}

\def\Hm{\H^{\fa}}

\def\am{a^{\fa}}

\def\mzv{multiple zeta value}

\shownotes

\def\ppc{supercongruence}
\def\Ppc{Supercongruence}

\def\MMo{\MM^{(1)}}

\def\pp#1{\pi_{1,#1}(X;t_{01},t_{10})}
\def\spe{\mathrm{Spec\,}}

\def\nc{\!\!\ll\!\! e_0,e_1\!\!\gg}

\begin{document}
\title[The completed finite period map and Galois theory]{The completed finite period map and Galois theory of supercongruences}
\author{Julian Rosen}
\email{julianrosen@gmail.com}
\date{\today}
\maketitle
\begin{abstract}
A period is a complex number arising as the integral of a rational function with algebraic number coefficients over a rationally-defined region. Although periods are typically transcendental numbers, there is a conjectural Galois theory of periods coming from the theory of motives. This paper formalizes an analogy between a class of periods called \mzv{}s, and congruences for rational numbers modulo prime powers (called \ppc{}s). We construct an analogue of the motivic period map in the setting of \ppc{}s, and use it to define a Galois theory of \ppc{}s. We describe an algorithm using our period map to find and prove \ppc{}s, and we provide software implementing the algorithm.
\end{abstract}


\section{Introduction}

\subsection{Periods}
A \emph{period} is a complex number given by the integral of a rational function with algebraic number coefficients, over a region in $\R^n$ defined by finitely many inequalities between polynomials with rational coefficients. Many familiar constants are periods, e.g.
\begin{gather*}
\pi=\iint\limits_{x^2+y^2\leq 1}\hspace{-1mm}1\,dx\,dy,\gap
\log(r)=\int_1^r\frac{dx}{x}\text{ for }r\in\Q_{>0},\\
\zeta(n)=\iiint\limits_{[0,1]^n}\frac{dx_1\ldots dx_n}{1-x_1\cdots x_n}\text{ for }n\geq 2.
\end{gather*}
The set $\PP\subset\C$ of all periods is a countable subring of $\C$ containing the algebraic numbers.
Although periods are typically transcendental numbers, the theory of motives predicts that a version of Galois theory should hold for periods (see \cite{And09}). The motivic Galois action has been studied in depth for a class of periods called \mzv{}s \cite{Bro12}, which we now define.

Recall that a \emph{composition} is a finite ordered list $\bs=(s_1,\ldots,s_k)$ of positive integers. The \emph{weight} and \emph{depth} of $\bs$ are $|\bs|=s_1+\ldots+s_k$ and $\ell(\bs)=k$, respectively. For $\bs$ a composition satisfying $s_1\geq 2$, we define the \emph{\mzv{}} by the convergent infinite series
\begin{equation}
\label{defmzv}
\zeta(\bs):=\sum_{n_1>\ldots>n_k\geq 1}\frac{1}{n_1^{s_1}\ldots n_k^{s_k}}\in\R.
\end{equation}
For every composition $\bs$, we have
\[
\zeta(\bs)=\iiint\limits_{0\leq x_1\leq\ldots\leq x_{|\bs|}\leq 1}\hspace{-4mm}\omega_1\ldots\omega_{|\bs|},
\]
where $\omega_i=dx_i/(1-x_i)$ if $i\in\{s_1,s_1+s_2,\ldots,s_1+\ldots+s_k\}$ and $\omega_i=dx_i/x_i$ otherwise. This iterated integral expression shows that \mzv{}s are periods.

\subsection{\Ppc{}s}
This paper develops a connection between periods and prime power divisibility properties of rational numbers. A \emph{\ppc} is a congruence between rational or $p$-adic numbers modulo a power of a prime $p$. We consider families of supercongruences holding for all primes at once, up to finitely many exceptions. For example, in 1979 \ape{} \cite{Ape79} proved that $\zeta(3)$ is irrational. The proof involved a sequence of rational approximations to $\zeta(3)$, whose denominators are given by the integers
\[
a_n:=\sum_{k=0}^n{n\choose k}^2{n+k\choose k}^2,
\]
now called the \emph{\ape{} numbers}. It is known \cite{Cho80} that the \ape{} numbers satisfy the \ppc{} $a_{p}\equiv 5\mod p^3$ for every prime $p\geq 5$. In our setting, we view the prime-indexed sequence $(a_p)$ as a finite analogue of a period.

Finite truncations of the \mzv{} series \eqref{defmzv} are called \emph{multiple harmonic sums}, and we write
\[
H_N(s_1,\ldots,s_k):=\sum_{N\geq n_1>\ldots>n_k\geq 1}\frac{1}{n_1^{s_1}\ldots n_k^{s_k}}\in\Q.
\]
Many \ppc{}s are known for multiple harmonic sums, especially in the case $N=p-1$, with $p$ a prime. Residue classes of multiple harmonic sums $\H(\bs)$ modulo $p$ (or sometimes modulo powers of $p$) are called finite multiple zeta values, and they have received considerable attention in recent years (e.g.\ see the recent works \cite{Koh14}\cite{Mur15}\cite{Mur16}\cite{Ono16}\cite{Oya15}\cite{Sai15}\cite{Sai16}\cite{Zha15}).

A useful technique for proving \ppc{}s is to relate terms to multiple harmonic sums. For example, consider the central binomial  coefficient ${2p\choose p}$, which has interesting arithmetic properties. It is not difficult to show that
\begin{equation}
\label{2pcp}
{2p\choose p}=2\sum_{n=0}^{\infty} p^n\H(\underbrace{1,\ldots,1}_n).
\end{equation}
Many series expansions related to \eqref{2pcp} are given by the author in \cite{Ros16a}.

\subsection{$p$-adic \mzv{}s}
The bridge between periods and \ppc{}s comes from a $p$-adic analogue $\zeta_p(\bs)\in\Q_p$ of the \mzv{}s. These $p$-adic numbers record the action of the crystalline frobenius on the motivic unipotent fundamental group of $\P^1\backslash\{0,1,\infty\}$ (\cite{Del05}, \S 5.28). There is a construction in terms of the Coleman integral due to Furusho \cite{Fur04,Fur07}. The $p$-adic \mzv{}s are expected to satisfy the same algebraic relations as the real \mzv{}s, along with the additional relation $\zeta_p(2)=0$.

A remarkable formula (stated as Theorem \ref{pj} below), conjectured by Hirose and Yasuda \cite{Yas14a} and independently discovered and proved by Jarossay \cite{Jar16} expresses the multiple harmonic sum $\H(\bs)$ as a $p$-adically convergent infinite linear combination of $p$-adic \mzv{}s. As a consequence, any $p$-adically convergent series involving multiple harmonic sums $\H(\bs)$ can we rewritten as series involving $p$-adic \mzv{}s. For example, \eqref{2pcp} can be used to derive a series representation
\begin{equation}
\label{cbi}
{2p\choose p}\equiv 2-4p^{3}\zeta_p(3)-12p^{5}\zeta_p(5)+4p^{6}\zeta_p(3)^{2}-36p^{7}\zeta_p(7)+\ldots
\end{equation}
in terms of $p$-adic \mzv{}s.
\subsection{Motivic periods}
\label{ssmgt}
The Galois theory of periods is conditional on algebraic independence statements for periods. To obtain an unconditional theory, one replaces $\PP$ by an abstractly-defined ring $\PPm$, called a ring of formal (or motivic) periods. One also defines a linear pro-algebraic group $G$, and an action of $G$ on $\PPm$. The motivic \emph{period map} is then a ring homomorphism $per:\PPm\to\C$ taking a motivic period to an actual (complex) period. The Grothendieck period conjecture for $\PPm$ is the statement that $per$ is injective, and if this is the case one gets an action of $G$ on $\PP$ (see \cite{Hub17}, Chapter 12).

The \mzv{}s are the periods of mixed Tate motives over $\Z$ \cite{Bro12}. Here the corresponding motivic period ring is $\M$, the \emph{ring of motivic \mzv{}s}. This is a commutative $\Q$-algebra spanned by elements $\zm(\bs)$. In addition to a period map $per\colon\M\to\R$ taking $\zm(\bs)$ to $\zeta(\bs)$,  for each prime $p$ there is also a $p$-adic period map $\tilde{per}_p:\M\to\Q_p$, taking $\zm(\bs)$ to $\zeta_p(\bs)$ \cite{Yam10}. While $per$ is conjectured to be injective, $\tilde{per}_p$ is not because it kills $\zm(2)$. If we write $\MM:=\M/\zm(2)\M$, the induced map
\[
per_p:\MM\to\Q_p
\]
is conjectured to be injective \cite{Yam10}. We write $\zmm(\bs)$ for the image of $\zm(\bs)$ in $\MM$.


\subsection{Results}
\subsubsection{The completed finite period map}
We construct an analogue $\per$ of the period maps $per$ and $per_p$ in the setting of finite periods. We need to consider infinite sums of powers of $p$ multiplied by $p$-adic multiple zeta values (for example \eqref{cbi}), so the domain of $\per$ is $\MM((T))$, the ring of formal Laurent series over $\MM$ in a variable $T$ (here $T$ acts as a formal version of an unknown prime $p$). The codomain of $\per$ is the quotient ring
\[
\Ai:=\frac{\left\{(a_p)\in\prod_p\Q_p:v_p(a_p)\text{ bounded below}\right\}}{\left\{(a_p)\in\prod_p\Q_p:v_p(a_p)\to\infty \text{ as }p\to\infty\right\}}.
\]
The rings $\MM((T))$ and $\Ai$ are complete with respect to a topologies arising from decreasing filtrations, given by
\begin{align*}
\fil^n \MM((T)) &= T^n\MM[[T]],\\
\fil^n\Ai &=\{(a_p)\in\Ai:\liminf v_p(a_p) \geq n\}.
\end{align*}
\begin{defn}
The \emph{completed finite period map} is the unique continuous ring homomorphism
\[
\per:\MM((T))\to\Ai
\]
satisfying $\per(\zmm(\bs))=(\zeta_p(\bs))$ and $\per(T)=(p)$.
\end{defn}

In \cite{Ros16a} the author introduces a subalgebra of $\Ai$ called the MHS algebra, consisting of elements that admit $p$-adic series expansion of a certain shape involving multiple harmonic sums, generalizing \eqref{2pcp}. A precise definition of the MHS algebra is given below as Definition \ref{defmhsalg}. The MHS algebra contains many ``elementary'' quantities (various sums of binomial coefficients, generalizations of the harmonic numbers, etc.), some of which are listed in Theorem \ref{thmhs} below.

The following result is Theorem \ref{propim} below.

\begin{theor}
The image of $\per$ is precisely the MHS algebra.
\end{theor}

We also formulate an analogue of the period conjecture, which says that $\per$ is compatible with the filtrations on its domain and codomain in the following strong sense.
\begin{conj}[Period conjecture]
\label{conjperintro}
The period map $\per$ satisfies
\[
\per^{-1}\lp\fil^n\Ai\rp=\fil^n\MM((T))
\]
for all $n$. In particular, $\per$ is injective.
\end{conj}

The truth of Conejcture \ref{conjperintro} would give a completely algorithmic way to prove \ppc{}s between elements of the MHS algebra. We describe the algorithm in Section \ref{secalg}. We also provide software implementing the algorithm, which we describe in Appendix A.

Another consequence of Conjecture \ref{conjperintro} would be a Galois theory of \ppc{}s. In Section \ref{secgalois} we define the Galois theory of \ppc{}s, and in Section \ref{secd1} we give some explicit computations. We apply the Galois theory to give an (unconditional) proof of a supercongruence for factorials.

%
%
\section{The completed finite period map}
\label{secalg}
In this section we construct the completed finite period map.

\subsection{$p$-adic \mzv{}s}

The unipotent fundamental groupoid of $X=\P^1\backslash\{0,1,\infty\}$ has the structure of a motivic groupoid \cite{Del02}. This means that for each $x$, $y\in X(\Q)$, there are various pro-algebraic varieties $\pi_{1,\bullet}(X;x,y)$, called \emph{realizations}, corresponding to different cohomology theories. We will make use of the de Rham, Betti, and crystalline realizations:
\begin{equation}
\label{paths}
\pi_{1,dR}(X;x,y),\gap \pi_{1,B}(X;x,y),\gap \pi_{1,p,crys}(X;x,y).
\end{equation}
Through the use of tangential basepoints, \eqref{paths} make sense when $x$ and $y$ are tangent vectors at $0$ or $1$. We write $t_{01}$ (resp.\ $t_{10}$) for the unit tangent vector at $0$ in the positive direction (resp.\ the unit tangent vector at $1$ in the negative direction).

For any commutative $\Q$-algebra $R$, we may identify the $R$-points of $\pp{dR}$ with the set of group-like elements of $R\nc$, the Hopf algebra of formal power series in non-commuting variables $e_0$ and $e_1$. The straight line path from $0$ to $1$ determines an element $dch\in\pp{B}(\Q)$, and under the Betti-de Rham comparison isomorphism
\[
C_{B,dR}:\pp{B}\times\spe\R\xrightarrow{\sim}\pp{dR}\times\spe\R,
\]
$dch$ maps to an element of $\pp{dR}(\R)$. In power series associated with $C_{B,dR}(dch)$, the coefficient of
\[
e_0^{s_1-1}e_1\ldots e_0^{s_k-1}e_1
\]
is $(-1)^k\zeta(s_1,\ldots,s_k)$.

Every vector bundle on $X$ with unipotent connection has a canonical trivialization, which determines an element $\gamma\in\pp{dR}(\Q)$. Via the de Rham-crystalline comparison isomorphism, the crystalline frobenius gives an automorphism $\phi_p$ of 
\[
\pp{p,crys}\simeq\pp{dR}\times\spe\Q_p.
\]
The image of $\gamma$ under $\phi_p$ is an element of $\pp{dR}(\Q_p)$. For each composition $\bs=(s_1,\ldots,s_k)$, the $p$-adic \mzv{} $\zeta_p(\bs)$ is defined so that the coefficient of
\[
e_0^{s_1-1}e_1\ldots e_0^{s_k-1}e_1
\]
in $\phi_p(\gamma)$ is $p^{|\bs|}\zeta(\bs)$. It is a result of Yasuda that for every composition $\bs$,
\begin{equation}
\label{int}
\zeta_p(\bs)\in\sum_{\ell\geq 0}\frac{p^\ell}{(|\bs|+\ell)!}\Z_p.
\end{equation}

The following formula expresses the multiple harmonic sum $\H(\bs)$ as an infinite series involving $p$-adic \mzv{}s.

\begin{theorem}[\cite{Jar16}, Theorem 1.2]
\label{pj}
For every composition $\bs=(s_1,\ldots,s_k)$, there is a $p$-adically convergent infinite series identity
\begin{gather}
\label{eqjar}
\hspace{-40mm}\H(\bs)=\sum_{i=0}^k \sum_{\ell_1,\ldots,\ell_i\geq 0}(-1)^{s_1+\ldots+s_i} \prod_{j=1}^i \frac{(s_j)_{\ell_j}}{\ell_j!}\cdot\hfill\\
\hspace{45mm}\hfill p^{\ell_1+\ldots+\ell_i}\zeta_p(s_i+\ell_i,\ldots,s_1+\ell_1)\zeta_p(s_{i+1},\ldots,s_k).
\end{gather}
Here $(s)_{\ell}:=s(s+1)\cdots(s+\ell-1)$ is the Pochhammer symbol.
\end{theorem}

\noindent Note: our normalization for $\zeta_p(\bs)$ differs from that in \cite{Jar16} by a factor of $p^{|\bs|}$.

\subsection{The target ring}
In \cite{Ros13}, the author defined a finite analogue of the multiple zeta function. The codomain of this map was a complete topological ring related to the finite adeles. Here we consider this ring with $p$ inverted. Define\footnote{The integral version of this ring was denoted $\hat{\cA}$ in \cite{Ros13}. Here we adopt the notation of \cite{Jar16a} to avoid conflict with the use of $\cA$ for the ring of motivic \mzv{}s modulo $\zeta(2)$.}
\[
\Ai:=\frac{\left\{(a_p)\in\prod_p\Q_p:v_p(a_p)\text{ bounded below}\right\}}{\left\{(a_p)\in\prod_p\Q_p:v_p(a_p)\to\infty \text{ as }p\to\infty\right\}}.
\]
We equip $\Ai$ with a decreasing, exhaustive, separated filtration:
\[
\fil^n\Ai=\left\{(a_p)\in\Ai: \liminf v_p(a_p)\geq n\right\}.
\]
The subsets $\fil^n\Ai\subset\Ai$ form a neighborhood basis of $0$ for a topology, making $\Ai$ into a topological ring. This ring is complete because it is the quotient  of the complete, first countable ring $\prod_p\Q_p$ with the uniform topology (i.e., the sets $\prod_p p^n\Z_p$ are a neighborhood basis of $0$) modulo a closed ideal.

Convergence in $\Ai$ is distinct from $p$-adic convergence. A sequence $(a_{p,1})$, $(a_{p,2})$, $\ldots$ converges to $(b_p)\in\Ai$ if and only if for every positive integer $m$, there exists $N=N(m)$ such that for all $n>N$ the \ppc{}
\[
a_{p,n}\equiv b_p\mod p^m\Z_p\gap
\]
holds for all but finitely many $p$. The finite set of primes for which the congruence fails may depend on $n$, so convergence in $\Ai$ does \emph{not} imply that $a_{p,n}$ converges $p$-adically to $b_p$ for any $p$ at all. An example is
\[
a_{p,n}=\begin{cases}1:p\leq n,\\0:p>n.\end{cases}
\]
Then $a_{p,n}\to 1$ for every $p$, but $(a_{p,n})\to 0$ in $\Ai$ (and in fact $(a_{p,n})$ is the constant sequence $0$).

The ring $\Ai$ is non-Archimedean, so an infinite series converges if and only if the terms go to $0$ (that is, for every integer $n$, all but finite many terms in the series are in $\fil^n\Ai$). Concretely, if $(a_{p,n})\to 0$, we have $\sum_n (a_{p,n})=(a_p)$, where
\[
a_p=\hspace{-10mm}\sum_{\substack{n\\ v_p(a_{p,n})\geq \liminf_\ell v_\ell(a_{\ell,n})}}\hspace{-10mm} a_{p,n}.
\]

\subsection{The period map}
Let $\M$ be the ring of motivic \mzv{}s (see \cite{Bro12}). It is a commutative $\Q$-algebra spanned by elements $\zm(\bs)$, called motivic \mzv{}s. We consider the quotient $\MM:=\M/\zm(2)\M$, and write $\zmm(\bs)$ for the image of $\zm(\bs)$ in $\MM$. For each prime $p$, there is a $p$-adic period map $per_p:\MM\to\Q_p$, taking $\zmm(\bs)$ to $\zeta_p(\bs)$.

We construct a new period map, with target $\Ai$. The domain of our period map is the ring of formal Laurent series
\[
\MM((T)) :=\left\{\sum_{k=-N}^\infty z_k T^k:N\in\Z, z_k\in\MM\right\}
\]
over $\MM$. Like $\Ai$, the ring $\MM((T))$ is complete with respect to the decreasing filtration
\[
\fil^n\MM((T)):=T^n\MM[[T]].
\]

The integrality result \eqref{int} implies that $\zeta_p(\bs)\in\Z_p$ for $p>|\bs|$, so that $(p^k\zeta_p(\bs))\in\fil^k\Ai$ for every $k$. This means infinite linear combinations of terms $(p^k\zeta_p(\bs))$ with rational coefficients converge in $\Ai$ provided that the terms satisfy $k\to\infty$.
\begin{definition}[Period map]
\label{defperiod}
The \emph{completed finite period map} is the continuous ring homomorphism
\begin{align}
\per:\MM((T))&\to\Ai,\\
\sum_{k=-N}^\infty z_k T^k&\mapsto\sum_{k=-N}^\infty \lp p^k per_p(z_k)\rp.\label{zpseries}
\end{align}
\end{definition}
The map $\per$ take $\fil^n\MM((T))$ into $\fil^n\Ai$. We expect that $\per$ is compatible with the filtrations in a stronger sense, which would be an analogue of the Grothendieck period conjecture. The conjecture takes the following form.
\begin{conjecture}[Period conjecture]
\label{conjper}
For every integer $n$,
\[
\per^{-1}\lp\fil^n\Ai\rp=\fil^n\MM((T)).
\]
In particular, $\per$ is injective.
\end{conjecture}

Conjecture \ref{conjper} is equivalent to the statement that for every non-zero element
\[
\sum_{i=1}^n c_i\zmm(\bs_i)\neq 0\in\MM,
\]
the $p$-adic number
\[
\sum_{i=1}^n c_i\zeta_p(\bs_i),
\]
which a priori is an element of $\Z_p$ for all $p$ sufficiently large, is actually in $\Z_p^\times$ for infinitely many $p$.

\begin{definition}
\label{defmotmhs}
Let $\bs$ be a composition. The \emph{motivic multiple harmonic sum} $\Hm(\bs)$ is the following element of $\MM((T))$:
\begin{equation}
\label{motmhs}
\sum_{i=0}^k \sum_{\ell_1,\ldots,\ell_i\geq 0} \prod_{j=1}^i {-s_j\choose \ell_j}\zmm(s_i+\ell_i,\ldots,s_1+\ell_1)\zmm(s_{i+1},\ldots,s_k)T^{\sum \ell_i}.
\end{equation}
\end{definition}

Formula \eqref{eqjar} implies that
\[
\per(\Hm(\bs))=\lp\H(\bs)\rp\in\Ai.
\]

\section{The MHS algebra}
A useful technique for proving \ppc{}s, used by the author in \cite{Ros16a}, is to express a quantity that appears in terms of multiple harmonic sums. For example, consider the hypergeometric sum
\[
S_n:=\sum_{k=0}^n {n\choose k}^4.
\]
It is known that $S_n$ satisfies a recurrence relation of length $3$, which can be found with Zeilberger's algorithm. For $n=p$, we can compute
\begin{align}
S_p&=2+\sum_{k=1}^{p-1}{p\choose k}^4\\
&=2+\sum_{k=1}^{p-1} \frac{p^4}{k^4}\left[\lp1-\frac{p}{1}\rp\cdots\lp1-\frac{p}{k-1}\rp\right]^4\\
&=2+\sum_{k=1}^{p-1} \frac{p^4}{k^4}\left[\sum_{n\geq 0}(-1)^n p^n H_{k-1}(1^n)\right]^4.\label{fourth}
\end{align}
A product of multiple harmonic sums with the same limit of summation can be written as a linear combination of multiple harmonic sums with the same limit, for example $H_n(1)H_n(2)=H_n(1,2)+H_n(2,1)+H_n(3)$.
This is the so-called \emph{series shuffle} or \emph{stuffle} product. If we expand out the fourth power in \eqref{fourth} this way and concatenate a $4$ at the beginning of each of the resulting compositions, we find that there are integer coefficients $c_1,c_2,\ldots$ and compositions $\bs_1,\bs_2,\ldots$ with $|\bs_i|\to\infty$, such that
\begin{equation}
\label{exexp}
S_p=\sum_{i=1}^\infty c_i p^{|\bs_i|}\H(\bs_i).
\end{equation}
It is not hard to compute the first few terms to obtain
\begin{equation}
\label{eqsp}
S_P\equiv 2+p^4\H(4)-4p^5\H(4,1)\mod p^6.
\end{equation}
Now we can, for example, combine \eqref{eqsp} with results of \cite{Zha08} to get a \ppc{} for $S_p$ in terms of Bernoulli numbers:
\[
S_p\equiv 2-\frac{16}{5}p^5B_{p-5}\mod p^6.
\]

We can combine \eqref{exexp} with \eqref{eqjar} to obtain an expression for $S_p$ in terms of $p$-adic \mzv{}s. The expression has many terms, but can be simplified using known relations among $p$-adic \mzv{}s to obtain a series that starts
\begin{gather}
\label{thingy}
S_p=2 - 16p^{5}\zeta_p(5) - 20p^{6}\zeta_p(3)^{2} - 143p^{7}\zeta_p(7)\\
 - p^{8}\left(456\zeta_p(3)\zeta_p(5)+ \frac{696}{5}\zeta_p(5, 3)\right) + O(p^{9}).
\end{gather}

Expansions like \eqref{exexp} are possible in a variety of other cases. The author makes the following definition in \cite{Ros16a}:
\begin{definition}
\label{defmhsalg}
The \emph{MHS algebra} is the subset $\mhs\subset\Ai$ consisting of elements $(a_p)$ such that there exist rational numbers $c_1,c_2,\ldots$, integers $b_1,b_2,\ldots$ going to infinity, and compositions $\bs_1,\bs_2,\ldots$, all independent of $p$, such that
\begin{equation}
\label{eqsumA}
(a_p)=\sum_{i=1}^\infty \lp c_i p^{b_i}\H(\bs_i)\rp\in\Ai.
\end{equation}
Concretely, this is equivalent to the condition that for every integer $n$, the congruence
\begin{equation}
\label{eqss}
a_p\equiv \sum_{\substack{i=1\\b_i<n}}^\infty c_i p^{b_i}\H(\bs_i)\mod p^n
\end{equation}
holds for all sufficiently large $p$ (note that the sum of the right hand side of \eqref{eqss} is finite).
\end{definition}
We sometimes abuse notation slightly and say that a quantity $a_p$ (depending on $p$) is in the MHS algebra when we mean $(a_p)\in\Ai$ is in the MHS algebra. The computation \eqref{fourth} shows that the hypergeometric sum $S_p$ is in the MHS algebra.

For each integer $n$, the quotient $\Ai/\fil^n\Ai$ is has the cardinality of the continuum. However, modulo $\fil^n\Ai$, an element of the MHS algebra can be described by a finite amount of data, hence the image of the MHS algebra in $\Ai/\fil^n\Ai$ is countable. It perhaps surprising, then, to find that many familiar elementary quantities are in the MHS algebra. The following theorem records a few examples.


\begin{theorem}[\cite{Ros16a}]
\label{thmhs}
The following quantities are in the  MHS algebra:
\begin{itemize}
\item the multiple harmonic sum $\displaystyle H_{f(p)}(\bs)$, where $f(x)\in\Z[x]$ has positive leading coefficient and $\bs$ is a composition,
\item the $p$-restricted multiple harmonic sum
\[
H_{f(p)}^{(p)}(\bs):=\sum_{\substack{f(p)\geq n_1>\ldots>n_k\geq 1\\p\nmid n_1\cdots n_k}}\frac{1}{n_1^{s_1}\cdots n_k^{s_k}},
\]
where $f(x)\in\Z[x]$ has positive leading coefficient,
\item the binomial coefficient $\displaystyle {f(p)\choose g(p)}$, for fixed polynomials $f(x),g(x)\in\Z[x]$ with positive leading coefficient,
\item the \ape{} numbers $a_{p-1}$ and $a_p$, where
\[
a_{N}:=\sum_{n=0}^{N}{N\choose n}^2{N+n
\choose n}^2.
\]
\item For fixed $r$, $k$, the sum
\[
\sum_{\substack{n_1+\ldots+n_k=p^r\\p\nmid n_1\cdots n_k}}\frac{1}{n_1\ldots n_k}.
\]
\end{itemize}

\end{theorem}

We prove the following result. 
\begin{theorem}
\label{propim}
The image of $\per$ is exactly the MHS algebra.
\end{theorem}
\begin{proof}
Suppose $(a_p)$ is in the MHS algebra, and let $c_i$, $b_i$, and $\bs_i$ be as in Definition \ref{defmhsalg}. Then the infinite sum
\[
a^\fa:=\sum_{i=1}^\infty c_i T^{b_i}\Hm(\bs_i)
\]
converges in $\MM((T))$, and we have $\per(a^{\fa})=(a_p)$.

Conversely, let
\[
a^{\fa}:=\sum_{i=1}^\infty c_i T^{b_i}\zmm(\bs_i)
\]
be an arbitrary element of $\MM((T))$, with $c_i\in\Q$, $b_i\in\Z$ with $b_i\to\infty$, and $\bs_i$ compositions. Yasuda \cite{Yas14b} has shown that the sum of the terms in \eqref{eqjar} with all $\ell_i=0$, ranging over all compositions $\bs$, generate the space of \mzv{}s modulo $\zeta(2)$. Jarossay (\cite{Jar16a}, Proposition 3) has shown this implies that for each $i$, there exist rational coefficients $d_{i,1},d_{i,2},\ldots$ and compositions $\bt_{i,1},\bt_{i,2},\ldots$ with $|\bs_{i,j}|\geq |\bs_i|$ and $|\bs_{i,j}|\to\infty$, such that
\[
\zmm(\bs_i)= \sum_{j=1}^\infty d_{i,j} T^{|\bt_{i,j}|-|\bs_i|}\Hm(\bt_{i,j}).
\]
Then we have
\[
\per(a^{\fa})=\sum_{i,j} \lp c_id_{i,j} p^{b_i+|\bt_{i,j}|-|\bs_i|}\H(\bs_{i,j})\rp\in\Ai,
\]
so that $\per(a^{\fa})$ is in the MHS algebra.
\end{proof}

\begin{remark}
In \cite{Ros16a}, we also consider the \emph{weighted MHS algebra}, which consists of elements $(a_p)\in\Ai$ for which there is an expansion \eqref{eqsumA} satisfying $b_i=|\bs_i|$ for all $i$. The weight-adic completion $\MMh$ of $\MM$ admits a diagonal embedding
\begin{align*}
\Delta:\MMh&\hookrightarrow\MM((T))\\
\zmm(\bs)&\mapsto \zmm(\bs)T^{|\bs|}.
\end{align*}
The proof of Theorem \ref{propim} also shows that the image of $\per\circ\Delta$ is the weighted MHS algebra.
\end{remark}

\subsection{Motivic lifts}
\label{secmot}
\begin{definition}
Suppose $a_p$ is a quantity depending on $p$. A \emph{motivic lift} of $a_p$ is an element $a^{\fa}\in\MM((T))$ such that $\per(a^{\fa})=(a_p)$.
\end{definition}
Proposition \ref{propim} implies an element of $\Ai$ admits a motivic lift if and only if it is in the MHS algebra. Conjecture \ref{conjper} would imply that motivic lifts are unique.

The motivic multiple harmonic sum $\Hm(\bs)$ is a motivic lift of $\H(\bs)$.
We can use $\Hm(\bs)$ to write down motivic lifts of other elements of the MHS algebra. For example, if $k\geq r\geq 0$ are integers, we have an expression for the binomial coefficient.
\[
{kp\choose rp}={k\choose r}\frac{\displaystyle\prod_{n=k-r}^{k-1}  \lp\sum_{i\geq 0}n^i p^i\H(1^i) \rp   }{\displaystyle\prod_{n=0}^{r-1}  \lp\sum_{i\geq 0}n^ip^i\H(1^i) \rp }.
\]
We define the \emph{motivic binomial coefficient} to be
\begin{equation}
\label{eqmb}
{kp\choose rp}^{\fa}:={k\choose r}\frac{\displaystyle\prod_{n=k-r}^{k-1}  \lp\sum_{i\geq 0}n^i\Hm(1^i)T^i \rp   }{\displaystyle\prod_{n=0}^{r-1}  \lp\sum_{i\geq 0}n^i\Hm(1^i)T^i \rp }\in\MM((T)).
\end{equation}
Note that each factor in the denominator is an element of $\MM[[T]]$ with constant term $1$, so is invertible.

In some cases it is difficult to write down a motivic lifts in closed form, but we can compute arbitrarily many terms with the aid of a computer. This is the case for a class of nested sums of binomial coefficients. We illustrate with an example. Suppose we encounter the quantity
\begin{equation}
\label{eqexample}
\sum_{p-1\geq n\geq m\geq 1}{p+m\choose m}{p\choose n}^2
\end{equation}
in a \ppc{}. Through some manipulations, it can be shown that \eqref{eqexample} is in the MHS algebra, though the expression is messy. With the help of a computer we find that the first few terms of a motivic lift of \eqref{eqexample} are
\begin{gather*}
3\zeta^{\mathfrak{a}}(3)T^{3}+2\zeta^{\mathfrak{a}}(3)T^{4}+\left(-2\zeta^{\mathfrak{a}}(3)+\frac{53}{2}\zeta^{\mathfrak{a}}(5)\right)T^{5}\\
+\left(2\zeta^{\mathfrak{a}}(3)+17\zeta^{\mathfrak{a}}(5)- \frac{1}{2}\zeta^{\mathfrak{a}}(3)^{2}\right)T^{6}+O(T^{7}).
\end{gather*}
In other words we have
\begin{gather*}
\sum_{p-1\geq n\geq m\geq 1}{p+m\choose m}{p\choose n}^2\equiv 3p^{3}\zeta_p(3)+2p^{4}\zeta_p(3)+p^{5}\left(-2\zeta_p(3)+\frac{53}{2}\zeta_p(5)\right)\\+p^{6}\left(2\zeta_p(3)+17\zeta_p(5)- \frac{1}{2}\zeta_p(3)^{2}\right)\mod p^7.
\end{gather*}

To be explicit, the general recipe to write down a motivic lift of an element of the MHS algebra is as follows: given an element of the MHS algebra
\[
(a_p)=\sum_{i=1}^\infty \lp c_i p^{b_i}\H(\bs_i)\rp\in\Ai,
\]
a motivic lift of $(a_p)$ is given by
\[
a^{\fa}:=\sum_{i=1}^\infty c_i \Hm(\bs_i)T^{b_i}\in\MM((T)).
\]

\section{An algorithm for proving \ppc{}s}
\label{secalg}
The author's work \cite{Ros16a} gives an algorithm for finding and proving supercongruences between elements of the MHS algebra. Here, we state a version of this algorithm using motivic lifts.

Suppose we would like to prove a \ppc
\begin{equation}
\label{toprove}
a_p\equiv a'_p\mod p^n,
\end{equation}
where  $a_p$ and $a'_p$ are in the MHS algebra.
\begin{enumerate}
\item We can find positive integers $k$, $k'$, rational coefficients $c_1,\ldots,c_k$ and $c'_1,\ldots,c'_{k'}$, integers $b_1,\ldots,b_k$ and $b'_1,\ldots,b'_{k'}$, and compositions $\bs_1,\ldots,\bs_k$ and $\bs'_1,\ldots,\bs'_{k'}$ such that for all sufficiently large $p$,
\[
a_p\equiv \sum_{i=1}^k c_i p^{b_i}\H(\bs_i),\gap a'_p\equiv\sum_{i=1}^{k'}c'_i p^{b'_i}\H(\bs'_i)\mod p^n.
\]
\item Use \eqref{motmhs} to compute
\begin{equation}
\label{mott}
\sum_{i=1}^k c_i T^{b_i}\Hm(\bs_i),\gap \sum_{i=1}^{k'}c'_i T^{b'_i}\Hm(\bs'_i)\in\MM((T))
\end{equation}
modulo $T^n$. This is a finite computation, and modulo $T^n$ there are unique representatives $a$ and $a'$ of \eqref{mott} living in $\MM[T,T^{-1}]$ whose degree in $T$ is at most $n-1$.
\item Use some source of relations between motivic \mzv{}s (e.g.\ use the \mzv{} data mine \cite{Blu10}) to check whether $a=a'$. If so, we have found a proof of \eqref{toprove}. If not, the truth of Conjecture \ref{conjper} would imply that \eqref{toprove} fails for infinitely many $p$.
\end{enumerate}
We provide software implementing this algorithm, which is described in Appendix A.

\section{Galois theory of supercongreunces}
\label{secgalois}
The motivic Galois group $\Gamma$ of the category of mixed Tate motives over $\Z$ is an affine pro-algebraic group defined over $\Q$, which acts on the ring $\MM$. We let $\Gamma$ fix $T$ to get an action of $\Gamma$ on $\MM((T))$. The truth of Conjecture \ref{conjper} would imply the action descends to the MHS algebra. Specifically, the action of $\Gamma$ on the MHS algebra is computed as follows.
\begin{itemize}
\item Given an element $(a_p)$ of the MHS algebra and $g\in \Gamma(\Q)$, find a motivic lift $\am\in\MM((T))$,
\item let $g$ act on $\am$ to get $g\circ \am\in\MM((T))$, and
\item apply $\per$ to get back an element of the MHS algebra.
\end{itemize}
In the absence of Conjecture \ref{conjper}, the result could depend on the choice of lift $\am$.

There are some computational challenges.
\begin{enumerate}
\item Given an element of the MHS algebra, it is sometimes difficult to write down an expansion in terms of MHS in closed form.

\item The group $\Gamma$ does not have a distinguished coordinate system, so describe the action of $\Gamma$ on $\MM$ requires many arbitrary choices.

\item The result of the group action is an element of $\Ai$ defined by an infinite linear combination of $p$-adic multiple zeta values, and it is not obvious whether this linear combination can be described combinatorially, e.g.\ as a sequence of rational numbers defined by nested sums.
\end{enumerate}

\section{Computations in depth 1}
\label{secd1}
Here, we explicitly compute the Galois action on elements of the MHS algebra in depth $1$, that is, elements for which an expansion \eqref{eqsumA} may be chosen with $\ell(\bs_i)\leq 1$ for all $i$. Equivalently these are the elements whose lift in $\MM$ may be chosen to be depth $1$. We write $\MMo$ for the subalgebra of $\MM$ generated by depth $1$ motivic \mzv{}s. As an algebra, $\MMo$ is freely generated by $\{\zmm(k):k\geq 3\text{ odd}\}$.

\subsection{The group}

The motivic Galois group for mixed Tate motives decomposes as a semi-direct product
\[
\Gamma:=\G_m\ltimes G_{\cU},
\]
where $G_{\cU}$ is a free pro-unipoten with one generator $\delta_k$ of degree $-k$ for each odd $k\geq 3$. To describe an action of $\Gamma$ on a ring $R$ is equivalent to the following:
\begin{itemize}
\item an action of $\G_m$ on $R$, or equivalently, a $\Z$-grading on $R$, and
\item a collection of locally nilpotent derivations $\{\delta_k:k\geq 3\text{ odd}\}$, such that $\delta_k$ reduces degrees by $k$.
\end{itemize}
The action is constructed \cite{Bro12a}. The generators $\delta_k$ are not canonical, but their images in the abelianization of $\Gamma$ are canonical. The action of $\Gamma$ on $\MMo$ factors through the abelianization, so we will not need to make any choices.

We describe the action of $\Gamma$ on $\MMo((T))$. The grading is by weight, so that $\zmm(k)T^n$ has degree $k$. The derivations $\delta_k$ are $\Q((T))$-linear, and satisfy 
\[
\delta_k\zmm(m)=\begin{cases}1\text{ if $k=m$},\\0\text{ otherwise.}\end{cases}
\]
In other words, with respect to the algebra basis $\{\delta_k:k\geq 3\text{ odd}\}$, $\delta_k$ acts by partial differentiation with respect to $\zmm(k)$.
For $r\in\Q^\times=\G_m(\Q)$ and $x\in\MM((T))$, we write $r\circ x$ for the image of $x$ under the action of $r$. The grading is concentrated in non-negative degrees, so projection onto degree $0$ part is a ring homomorphism. We denote the degree $0$ part of an  element $x\in\MM((T))$ by $x_0$. The $\G_m$ action extends to an action of the multiplicative monoid $\G_m\cup\{0\}$, and $x_0$ is the image of $x$ under $0\in\G_m\cup\{0\}$. Thus if we have an algebraic formula for $r\circ x$, we obtain $x_0$ by setting $r=0$.

\subsection{Power sums}
First we consider the power sum multiple harmonic sums $\H(n)$.
\begin{proposition}
\label{propgmpo}
For positive integers $n$ and $r$, we have
\[
\per(r\circ\Hm(n))=r^nH_{rp}^{(p)}(n):=r^n \sum_{\substack{j=1\\p\nmid j}}^{rp}\frac{1}{j^n}.
\]
In addition we have $\Hm(n)_0=0$.
\end{proposition}
\begin{proof}
The motivic power sum is
\begin{equation}
\label{HtoZ}
H_{p-1}^\fa(n)=(-1)^n \sum_{k\geq 1} {n+k-1\choose n-1}\zmm(n+k)T^k.
\end{equation}
We act by $r$ and apply $\per$ to obtain
\[
\per(r\circ H^\fa_{p-1}(n))=(-1)^n \sum_{k\geq 1} {n+k-1\choose n-1}r^{n+k}p^k\zeta_p(n+k).
\]
It follows from \cite{Was98} (Theorem 1) that the right hand side above is
\[
r^n \sum_{\substack{j=1\\p\nmid j}}^{rp}\frac{1}{j^n}.
\]

The second part of the claim is immediate from \eqref{HtoZ}.
\end{proof}

Next we compute the derivations.
\begin{proposition}
\label{propdpo}
For positive integer $n$, $k$, with $k\geq 3$ odd, we have
\[
\delta_k \Hm(n)=\begin{cases}
(-1)^n T^{k-n}{k-1\choose n-1}&:\gap k> n,\\
0&:\gap k \leq n.\end{cases}
\]
\end{proposition}
\begin{proof}
Immediate from \eqref{HtoZ}.
\end{proof}

\subsection{Elementary symmetric sums}
Here we consider the elementary symmetric multiple harmonic sums
\[
\H(1^n):=\H(\underbrace{1,\ldots,1}_n).
\]

\begin{remark}
\label{remnew}
The association $\H(\bs)\mapsto\Hm(\bs)$ is a homomorphism for the series shuffle product (\cite{Jar16}, Theorem 1). For example, we have $\H(1)\H(2)=\H(1,2)+\H(2,1)+\H(3)$, and the corresponding identity $\Hm(1)\Hm(2)=\Hm(1,2)+\Hm(2,1)+\Hm(3)$ holds for the motivic lifts. Newton's formula relates the elementary symmetric and power sum symmetric functions, and these relations also hold for the elementary symmetric and power sum motivic multiple harmonic sums.
\end{remark}

\begin{proposition}
\label{propgmel}
For positive integers $n$, $r$, we have
\[
\per(r\circ\Hm(1^n))=r^n H_{rp}^{(p)}(1^n):=r^n\hspace{-5mm}\sum_{\substack{rp\geq m_1>\ldots>m_n\geq 1\\p\nmid m_1\ldots m_n}}\frac{1}{m_1\ldots m_n}.
\]
In addition we have $\Hm(1^n)_0=0$.
\end{proposition}
\begin{proof}
Newton's formula for symmetric functions shows we can write $\H(1^n)$ as a polynomial in the depth $1$ sums, independent of $p$. The result now follows from Proposition \ref{propgmpo} and Remark \ref{remnew}. The second claim follows similarly.
\end{proof}

We similarly get a ``finite'' formula for the derivations.
\begin{proposition}
\label{propes}
For positive integer $n$, $k$, with $k\geq 3$ odd, we have
\[
\delta_k\Hm(1^n)= -\hspace{-5mm}\sum_{i=\max(1,k-n)}^{k-1} \frac{1}{k}{k\choose i}T^i \Hm(1^{n-k+i}).
\]
\end{proposition}
\begin{proof}
The proof is induction on $n$. For $n=0$ this is true because both sides are $0$.

Suppose $n\geq 1$. Newton's formula for symmetric functions implies
\[
\Hm(1^n)=\frac{1}{n}\sum_{i=1}^n (-1)^{i-1}\Hm(i)\Hm(1^{n-i}).
\]
We compute
\begin{align*}
\delta_k \Hm(1^n) &= \frac{1}{n}\sum_{i=1}^n (-1)^{i-1}\delta_k\big[\Hm(i)\Hm(1^{n-i})\big]\\
&=\frac{1}{n}\sum_{i=1}^n (-1)^{i-1}\delta_k[\Hm(i)]\cdot \Hm(1^{n-i})\\
&\gap+\frac{1}{n}\sum_{i=1}^n (-1)^{i-1}\Hm(i)\delta_k\big[\Hm(1^{n-i})\big]\\
\end{align*}
The left-hand sum is
\begin{align*}
&-\frac{1}{n}\sum_{i=1}^{\min(n,k-1)}{k-1\choose i-1}T^{k-i}\Hm(1^{n-i})\\
&\hspace{15mm}=-\frac{1}{n}\sum_{j=\max(1,k-n)}^{k-1}{k-1\choose j}T^j\Hm(1^{n-k+j})
\end{align*}

The right-hand sum is
\begin{gather*}
\hspace{-20mm} -\frac{1}{nk}\sum_{i=1}^n(-1)^{i-1}\Hm(i)
\sum_{j=\max(1,k-n+i)}^{k-1} {k\choose j}p^j \Hm(1^{n-i-k+j})\\
\hspace{10mm}=-\frac{1}{nk}\sum_{j=\max(1,k-n+1)}^{k-1}{k\choose j}T^j \sum_{i=1}^{n-k-j} (-1)^{i-1}\Hm(i)\Hm(1^{n-k+j-i})\\
=-\frac{1}{nk}\sum_{j=\max(1,k-n+1)}^{k-1} {k\choose j}p^j(n-k+j)\Hm(1^{n-k+j}).
\end{gather*}
The result now follows by adding these two sums together and observing that the coefficient of $T^j\Hm(1^{n-k+n})$ is
\[
-\frac{1}{n}{k-1\choose j} -\frac{1}{nk}{k\choose j}(n-k+j) =\frac{-1}{k}{k\choose j}
\]
\end{proof}

\subsection{Binomial coefficients}
Next we compute the Galois action on the binomial coefficients ${ap\choose bp}$, which are in MHS algebra.

First, define
\[
c_n:={np\choose p}=n\sum_{j\geq 0}(n-1)^jp^j\H(1^j),
\]
with motivic lift
\begin{equation}
\label{cc}
c_n^{\fa}:=n\sum_{j\geq 0}(n-1)^jT^j\Hm(1^j)\in\MM[[T]].
\end{equation}
\begin{proposition}
\label{propgmc}
Fix a positive integer $r\in\G_m(\Q)$ and a positive integer $n$. Then
\[
\per(r\circ c_n^\fa)=\frac{n}{{rn\choose r}}{rnp\choose rp}.	
\]
In addition, we have $(c_n^\fa)_0=n$.
\end{proposition}
\begin{proof}
We compute
\begin{align*}
\per(r\circ c_n^\fa)&=n\sum_{j\geq 0}(n-1)^j(rp)^j H_{rp}^{(p)} (1^j)\\
&=n\prod_{\substack{i=1\\p\nmid i}}^{rp}\lp(1+\frac{(n-1)rp}{i}\rp\\
&=\frac{n}{{rn\choose r}}{rnp\choose rp}.	
\end{align*}
The second part of the claim follows from the expression for $c_n$ in terms of the elementary symmetric sums $\H(1^k)$.
\end{proof}
\begin{proposition}
For a positive integer $r\in\G_m(\Q)$ and $a\geq b\geq 0$, we have
\[
\per\lp r\circ{ap\choose bp}^\fa\rp = \frac{{a\choose b}}{{ra\choose rb}}{rap\choose rbp}.
\]
In addition, we have ${ap\choose bp}^\fa_0={a\choose b}$.
\end{proposition}
\begin{proof}
First we observe that
\[
{ap\choose bp}^\fa=\frac{c^\fa_a \cdots c^\fa_{a-b+1}}{c^\fa_b\cdots c^\fa_1}.
\]
It follows that
\begin{align*}
\per\lp r\circ {ap\choose bp}^\fa\rp&=\frac{\prod_{n=a-b+1}^a n{rnp\choose rp}/{rn\choose r}}{\prod_{n=1}^b n {rnp\choose rp}/{rn\choose r}}\\
&=\frac{{a\choose b}}{{ra\choose rb}}{rap\choose rbp}.
\end{align*}
The second part of the claim follows by setting $r=0$.
\end{proof}

We can also compute how the derivations act on binomial coefficients.
\begin{proposition}
\label{propdbin}
For positive integers $a$, $b$, and $k$, with $k\geq 3$ odd, we have
\[
\delta_k{ap\choose bp}^\fa=-\frac{T^k}{k}\big[a^k-b^k-(a-b)^k\big]{ap\choose bp}^\fa.
\]
\end{proposition}
\begin{proof}

Proposition \ref{propes} and \eqref{cc} imply
\begin{align*}
\delta_k c_n^\fa &= -\frac{n}{k}\sum_{j\geq 0} (n-1)^j T^j\sum_{i=\max(1,k-j)}^{k-1} {k\choose i}p^i \Hm(1^{j-k+i})\\
&=-\frac{n}{k}\sum_{m\geq 0} \Hm(1^m)\sum_{j=m+1}^{m+k-1}{k\choose j-m} (n-1)^jT^{k+m}\\
&=-\frac{np^k}{k}\sum_{m\geq 0} (n-1)^m\Hm(1^m)\sum_{j=1}^{k-1}{k\choose j} (n-1)^j \\
&=-\frac{T^k}{k}\big[n^k-(n-1)^k-1\big]c_n.
\end{align*}
Finally, we have
\begin{align*}
\delta_k{ap\choose bp}^\fa&={ap\choose bp}^\fa\cdot\left(\sum_{n=a-b+1}^a\frac{\delta_k c_n^\fa}{c_n^\fa}-\sum_{n=1}^b\frac{\delta_k c_n^\fa}{c_n^\fa}\right)\\
&=-\frac{T^k}{k}\big[a^k-b^k-(a-b)^k\big]{ap\choose bp}^\fa.
\end{align*}
\end{proof}

\subsection{Products of factorials}
For $b$ a positive integer, we do not expect $(bp)!$ is in the MHS algebra. However, if $b_1,\ldots,b_k$ are positive integers and $n_1,\ldots,n_k$ are arbitrary integers satisfying $\sum n_ib_i=0$, then the product
\begin{equation}
\label{prodfac}
\alpha:=\prod_{i=1}^k (b_ip)!^{n_i}
\end{equation}
is in the MHS algebra: we can divide the $i$-th term by $p!^{n_ib_i}$ to express \eqref{prodfac} as
\begin{equation}
\label{pfcs}
\prod_{i=1}^k (c_1c_2\ldots c_{b_i})^{n_i}.
\end{equation}
We get a motivic lift $\alpha^\fa$ by replacing each term $c_j$ with $c^\fa_j$.

\begin{proposition}
For $r$ a positive integer, we have
\[
\per(r\circ\alpha^\fa)=\prod_{i=1}^k \lp\frac{b_i!r!^{b_i}}{(rb_i)!} (rb_ip)!\rp^{n_i}.
\]
Additionally we have $\alpha^{\fa}_0=\prod_i b_i!^{n_i}$.
\end{proposition}
\begin{proof}
This is straightforward to derive from \eqref{pfcs} and Proposition \ref{propgmc}.
\end{proof}

\begin{proposition}
\label{propdf}
For $m\geq 3$ odd, we have
\[
\frac{\delta_m\alpha^{\fa}}{\alpha^{\fa}}=-\frac{T^m}{m} \sum_{i=1}^k n_i b_i^m.
\]
\end{proposition}

\begin{proof}
We compute
\begin{align*}
\frac{\delta_m\alpha^{\fa}}{\alpha^{\fa}}&= \sum_{i=1}^k n_i\lp\frac{\delta_m c^\fa_1}{c^\fa_1}+\cdots+\frac{\delta_m c^\fa_{b_i}}{c^\fa_{b_i}}\rp\\
&=-\frac{T^m}{m} \sum_{i=1}^k n_i\lp b_i^m-b_i\rp\\
&=-\frac{T^m}{m} \sum_{i=1}^k n_i b_i^m,\\
\end{align*}
where on the second line we observed that each summand telescopes, and on the last line we have used that $\sum_i n_ib_i=0$.
\end{proof}

\subsection{A supercongruence for factorials}
We end with an alternate proof of a known supercongruence for factorials. Our proof uses the derivations $\delta_k$.
\begin{theorem}[Granville \cite{Gra97}, Proposition 5]
\label{thex}
Suppose $b_1,\ldots,b_k\in\Z_{\geq 0}$, $n_1,\ldots,n_k\in\Z$, and assume that
\[
\sum_{i=1}^k n_i b_i^m=0
\]
for $m=1,3,5,\ldots,2n-1$. Then for all primes $p$ sufficiently large,
\begin{equation}
\label{pf}
\prod_{i=1}^k (b_ip)!^{n_i}\equiv \prod_{i=1}^k b_i!^{n_i}\mod p^{2n+1}.
\end{equation}
\end{theorem}

\begin{proof}
Let $\alpha^\fa$ be the motivic lift of the left hand side of \eqref{pf} coming from \eqref{pfcs}. By Proposition \ref{propdf}, we have $\delta_m\alpha^{\fa}=0$ for $m=3,5,\ldots,2n-1$. Proposition \ref{propdbin} implies $\delta_m\alpha^{\fa}\in\fil^{m}\MM((T))$ for all $m$, so we conclude $\delta_m\alpha^{\fa}\in\fil^{2n+1}\MM((T))$ for all $m$. It follows that
\[
\alpha^\fa\equiv \alpha_0=\prod_{i=1}^kb_i!^{n_i}\mod\fil^{2n+1}\MM((T)).
\]
We obtain the desired result by applying $\per$.
\end{proof}
\appendix

\section{Software}
We provide software for computing motivic lifts and for verifying \ppc{}s. The software is written in Python 2.7, and is available at
\begin{center}
\url{https://sites.google.com/site/julianrosen/mhs}.
\end{center}
The software can express multiple harmonic sums interms of $p$-adic multiple zeta values using a chosen basis.
\bigskip

\verb|>>> a = Hp(1,3,2)|

\verb|>>> a.disp()|

\bigskip
$H_{p-1}(1,3, 2)$
\bigskip

\verb|>>> a.mzv()|
\[
- \frac{9}{2}\zeta_p(3)^{2}+\frac{67}{16}p\zeta_p(7)+p^{2}\left(\zeta_p(5, 3)+\frac{23}{2}\zeta_p(3)\zeta_p(5)\right)+O(p^{3})
\]

The software will also compute expansions for other elements of the MHS algebra, for example the \ape{} numbers.

\verb|>>> a = aperybp()|

\verb|>>> a.disp()|

\bigskip
$\displaystyle \sum_{n=0}^{p-1}{p-1\choose n}^2{p+n-1\choose n}^2$
\bigskip

\verb|>>> a.mzv(err=8)|
\[
1+2p^{3}\zeta_p(3)-16p^{5}\zeta_p(5)+4p^{6}\zeta_p(3)^{2}-100p^{7}\zeta_p(7)+O(p^{8})
\]

\noindent The optional argument \verb|err=8| says we only want an expansion modulo $p^8$. We can also compute expansions for $p$-restricted multiple harmonic sums whose limit of summation is polynomial in $p$.

\verb|>>> a = H_poly_pr([1,-2,2], (1,2))|

\verb|>>> a.disp()|

\bigskip
$\displaystyle H_{2 p^{2} - 2 p + 1}^{(p)}(1, 2)$
\bigskip

\verb|>>> a.mzv(err=4)|
\[
6\zeta_p(3)-10p\zeta_p(3)+p^{2}\left(-4\zeta_p(3)+27\zeta_p(5)\right)+p^{3}\left(- \frac{293}{3}\zeta_p(5)+8\zeta_p(3)^{2}\right)+O(p^{4})
\]

We can also deal with various other nested sums. For example, suppose we are interested in the sum
\[
\sum_{n=1}^{p-3}{p\choose n}^2\frac{1}{(n+1)(n+2)}.
\]

\verb|>>> a = (BINN(0,0)**2*nn(-1)*nn(-2)).sum(1,-2).e_p()|

\verb|>>> a.mzv()|
\[
- \frac{11}{8}p^{2}+p^{3}\left(\frac{35}{16}-2\zeta_p(3)\right)- \frac{93}{32}p^{4}+p^{5}\left(\frac{215}{64}+\frac{5}{2}\zeta_p(3)-6\zeta_p(5)\right)+O(p^{6})
\]

\ack We thank Jeffrey Lagarias for many helpful suggestions.

\bibliographystyle{halpha}
\bibliography{jrbiblio}
\end{document}